\documentclass[x11names,11pt,twoside]{amsart}

\usepackage[utf8]{inputenc}
\usepackage[english]{babel}
\usepackage{geometry, verbatim, url}
\usepackage{paralist, graphics, graphicx}
\usepackage{hyperref}
\usepackage{amsmath, amsthm, amssymb, amsfonts}
\usepackage{fancyhdr}
\usepackage{color}

\theoremstyle{plain}
\newtheorem{theorem}{Theorem}[]
\newtheorem{proposition}[theorem]{Proposition}

\newtheorem{lemma}[theorem]{Lemma}
\theoremstyle{definition}

\newtheorem{question}[theorem]{Question}

\newtheorem{remark}[theorem]{Remark}
\newtheorem{maintheorem}[]{Theorem}

\newtheorem{conjecture}[]{Conjecture}

\date{}

\title{Non-ridge-chordal complexes whose clique complex has shellable Alexander dual}

\author{Bruno Benedetti}
\address{Department of Mathematics, University of Miami, 
Coral Gables, FL 33146}
\email{bruno@math.miami.edu}

\author{Davide Bolognini}
\address{Dipartimento di Matematica, Universit\'a di Bologna, 
Bologna, Italia, 40126}
\email{davide.bolognini2@unibo.it}

\begin{document}

\begin{abstract} A recent conjecture that appeared in three papers by Bigdeli--Faridi,  Dochtermann, and Nikseresht, is that every simplicial complex whose clique complex has shellable Alexander dual, is  ridge-chordal. This strengthens the long-standing Simon's conjecture that the $k$-skeleton of the simplex is extendably shellable, for any $k$. We show that the stronger conjecture has a negative answer, by exhibiting an infinite family of counterexamples. 
\end{abstract}

\maketitle

\section{Introduction}
Shellability is a property satisfied by three important families of objects in combinatorics, namely, polytope boundaries \cite{Zie} (see also \cite{AB}) 
and order complexes of geometric lattices \cite{Bj}. Moreover, skeleta of shellable complexes are themselves shellable \cite{BjWa}.  \emph{Extendable shellability} is the stronger demand that any shelling of any full-dimensional subcomplex  may be continued into a shelling of the whole complex. This property is less understood than shellability, and much less common. It is easy to construct polytopes that are not extendably shellable \cite{Zie}. In 1994 Simon conjectured that, for any integer $0 \le  d \le n$, the $d$-skeleton of the $n$-simplex is extendably shellable \cite[Conjecture 4.2.1]{Simon}. For $d\le2$ this was soon proven by Bj\"orner and Eriksson \cite{BjEr}, but for $3 \le d \le n-3$ the conjecture remains open.

Recently Bigdeli et al. \cite{ur3} and Dochtermann et al. \cite{doc2} established Simon's conjecture for $d \ge n-2$,  showing also that shellability is equivalent to extendable shellability for $d$-complexes with up to $d+3$ vertices \cite{doc2}. Their approach is based on a higher-dimensional extension of the graph-theoretic notion of chordality, called \emph{ridge--chordality}, which we recall below. 
Given a $d$-dimensional pure simplicial complex $\Delta$, any $(d-1)$-dimensional face of it is called a {\em ridge}.  {\em ``Deleting above a ridge''} of $\Delta$ means to consider the simplicial complex whose facets are the facets of $\Delta$ not containing that ridge. A {\em clique} of $\Delta$ is any subset $V \subseteq [n]$ such that  all subsets of $V$ of size $d+1$ appear among the facets of $\Delta$.
For example, if $\Delta$ is the graph $\{12, 23, 13, 14\}$, then $1$, $12$ and $123$ are cliques, whereas $124$ and $1234$ are not. 

A pure $d$-dimensional simplicial complex $\Delta$ is called \emph{ridge-chordal} if $\Delta=\emptyset$ or if it can be reduced to the empty set by repeatedly deleting above a ridge $r$ such that the vertices of the star of $r$ form a clique \cite{ur}. One can see that ``ridge-chordal $1$-complexes'' are precisely the graphs admitting a perfect elimination ordering, i.e. graphs in which every minimal vertex cut is a clique; by Dirac's theorem, these are precisely the ``chordal graphs'', the graphs where every cycle of length at least four has a chord \cite{dir}.  

Now, let $\operatorname{Cl}(\Delta)$ be the ``clique complex'' of $\Delta$, i.e., the simplicial complex whose faces are the cliques of~$\Delta$.
This $\operatorname{Cl}(\Delta)$ is a simplicial complex of dimension at least $d$, with the same $d$-faces of $\Delta$ and the same $(d-1)$-faces of the $n$-simplex. 
The following conjecture appeared naturally, in  several recent works:

\begin{conjecture}[{\cite[Question 6.3]{bigfar}, \cite[Conjecture 4.8]{doc}, \cite[Statement A]{N}}] \label{conj:1} $\qquad$ \par
\noindent If the Alexander dual of $\operatorname{Cl}(\Delta)$ is shellable, then $\Delta$ is ridge-chordal.
\end{conjecture}

There are three reasons why Conjecture \ref{conj:1} is natural and of interest:
\begin{compactenum}[(1)]
\item As explained by Bigdeli et al.~\cite[Corollary 3.7]{ur3} and \cite[Corollary 4.16]{N}, Conjecture \ref{conj:1} directly implies Simon's conjecture, cf.~Remark \ref{cluttersapproach}. 
\item The conjecture is true if one slightly strengthens the assumption ``shellable'' into ``vertex-decomposable''. This fact is proven in the pure case by Nikseresht \cite[Theorem 3.10]{N}, and in full-generality by Bigdeli--Faridi \cite[Theorem 5.2]{bigfar}; see also Remark \ref{bigfarapproach} below. Also, Conjecture \ref{conj:1} holds for $\dim \Delta = 1$.
\item Some partial converse holds: If $\Delta$ is ridge-chordal, then the Alexander dual of $\operatorname{Cl}(\Delta)$ is Cohen--Macaulay over any field \cite[Theorem 3.2]{ur}, although not necessarily shellable or constructible \cite[Example 3.14]{ur}.
\end{compactenum}

The purpose of this short note is to strongly disprove Conjecture \ref{conj:1}: 

\begin{maintheorem}\label{main} For any $k \geq 2$ there is a constructible $2$-dimensional complex $\Delta_k$ that is not ridge-chordal, such that the Alexander dual of $\operatorname{Cl}(\Delta_k)$ is pure $(5k-2)$-dimensional, shellable, and even $4$-decomposable.
\end{maintheorem}

Theorem \ref{main} provides a non-trivial class of complexes that are $4$- but not $0$-decomposable (cf.~Remark~\ref{bigfarapproach}) in arbitrarily high dimension. This infinite family does not disprove Simon's conjecture, because the shelling of the Alexander dual of $\operatorname{Cl}(\Delta_k)$, which is $(5k-2)$-dimensional on $5k+2$ vertices, does extend to a shelling of the $(5k-2)$-skeleton of the $(5k+1)$-simplex, as we will see in Remark \ref{nocontro}.

\section{Construction of the counterexamples}
Recall that the link and the deletion of a face $\sigma \in \Delta$ are defined respectively by 
\[  \operatorname{link}_{\Delta}(\sigma):= \{\tau \in \Delta: \sigma \cap \tau = \emptyset,  \sigma \subseteq F \supseteq \tau \textrm{ for some facet } F\} \quad \textrm{ and } \quad
 \operatorname{del}_{\Delta}(\sigma):=\{\tau \in \Delta: \sigma \not \subseteq \tau\}.\]
We say that a face $\sigma$ in a pure simplicial complex $\Delta$ is {\em shedding} if $\operatorname{del}_{\Delta}(\sigma)$ is pure of dimension $\operatorname{dim} \Delta$. An equivalent formulation (see for instance \cite[Definition 3.1]{Wood}) is the following: $\sigma$ is shedding if and only if for every face $F \in \Delta$ such that $\sigma \subseteq F$ and for every $v \in \sigma$, there exists $w \notin F$ such that $(F \setminus \{v\}) \cup \{w\} \in \Delta$. A pure simplicial complex $\Delta$ is {\em $k$-decomposable} if $\Delta$ is a simplex or if there exists a shedding face $\sigma \in \Delta$ with $\dim \sigma \leq k$ such that $\operatorname{link}_{\Delta}(\sigma)$ and $\operatorname{del}_{\Delta}(\sigma)$ are both $k$-decomposable \cite{ProvanBillera}.  It is easy to see  that if $\Delta$ is $k$-decomposable then it is also $t$-decomposable, for every $k \leq t \leq \dim \Delta$. The notion of $k$-decomposable interpolates between vertex-decomposable complexes (which are the same as $0$-decomposable complexes) and shellable complexes (which are the same as $d$-decomposable complexes, where $d$ is their dimension). 

We start with a Lemma that is implicit in the work of Bigdeli--Faridi \cite{bigfar}. Recall that a {\em free face} in a simplicial complex $\Delta$ is a face strictly contained in only one facet of $\Delta$.

\begin{lemma}\label{lem:existence}
Let $r$ be a ridge of a pure $d$-dimensional simplicial complex $\Delta$, with $d \ge 1$. Let $S$ be the set of vertices of $\operatorname{Star}(r, \Delta)$. Then $S \in \operatorname{Cl}(\Delta)$ $ \Longleftrightarrow $  $r$ is a free face in $\operatorname{Cl}(\Delta)$.
\end{lemma}

\begin{proof}
\begin{compactdesc}
\item[$\Rightarrow$] If $r$ lies in two facets $F_1$ and $F_2$ of $\operatorname{Cl}(\Delta)$, then $F_i=r \cup S_i$ for some $S_i \subseteq [n]$. Since $F_1,F_2 \in \operatorname{Cl}(\Delta)$, for every $s \in S_1 \cup S_2$ we have $r \cup \{s\} \in \Delta$. So $r \cup (S_1 \cup S_2) \subseteq S$ is a clique of $\Delta$. Since $r \cup S_1$ and  $r \cup S_2$ are both facets of $\operatorname{Cl}(\Delta)$, we have $S_1=S_2$, whence $F_1=F_2$.
\item[$\Leftarrow$]  Let $F$ be the unique facet of $\operatorname{Cl}(\Delta)$ that contains $r$. Were there a vertex $s$ of $S$ outside $F$, we would have $r \cup \{s\} \in \Delta \subseteq \operatorname{Cl}(\Delta)$; so there would be $G \in \operatorname{Cl}(\Delta)$, $G \neq F$, such that $r \cup \{s\} \subseteq G$, a contradiction. Hence $S \subseteq F$ and $S \in \operatorname{Cl}(\Delta)$. 
\qedhere
\end{compactdesc}
\end{proof}

\begin{lemma}\label{lem:existence1} Let $\Delta$ be a pure simplicial complex. If $\Delta$ is ridge-chordal and $\dim \Delta = \dim \operatorname{Cl}(\Delta)$, then $\Delta$ has at least one free ridge.
\end{lemma}

\begin{proof}
If $\Delta$ is ridge-chordal, then it must have a ridge $r$ such that the vertices of $\operatorname{Star}(r, \Delta)$ form a clique. By Lemma \ref{lem:existence}, this $r$ is a free face of  $\operatorname{Cl}(\Delta)$. But since $\mathrm{dim} \Delta =\mathrm{dim} \operatorname{Cl}(\Delta)$, the set of free ridges of $\Delta$ coincides with the set of $(\operatorname{dim} \Delta-1)$-dimensional free faces of $\operatorname{Cl}(\Delta)$, because any ridge we add when passing from $\Delta$ to $\operatorname{Cl}(\Delta)$ belongs to no face of dimension equal to $\operatorname{dim} \Delta$. 
\end{proof}

\begin{figure}[htb]
  \centering
\includegraphics[width=0.5\linewidth]{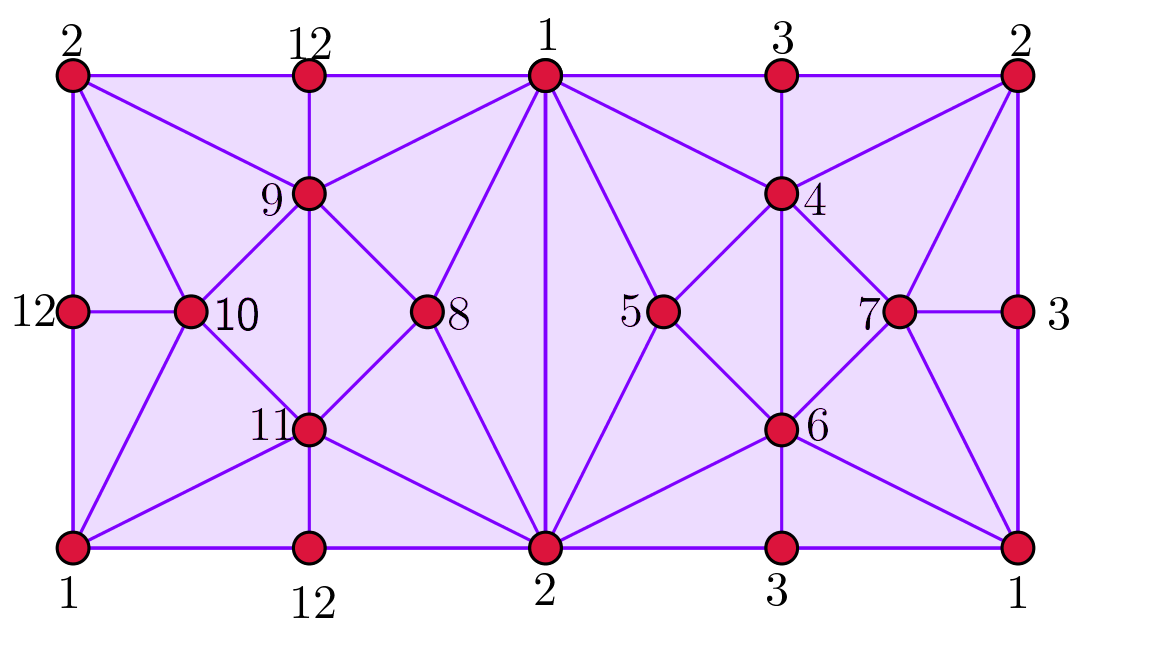} 
\vskip-2mm
  \caption{The constructible contractible $2$-complex $\Delta^2_2$ 
   without free edges, constructed  by Barmak in \cite[Example 11.2.9]{BA}, is not ridge-chordal: See  Proposition \ref{prop:firstpart} below. }
\end{figure}

Recall that a pure $d$-dimensional complex $\Delta$ with $N$ facets is \emph{constructible} if either $d(N-1)=0$, or if $\Delta$ splits as $\Delta = \Delta_1 \cup \Delta_2$, with $\Delta_1$, $\Delta_2$ constructible $d$-dimensional complexes and $\Delta_1 \cap \Delta_2$ constructible $(d-1)$-dimensional. All shellable complexes are constructible, but the converse is false, cf.~e.g.~\cite[Prop.~6.7]{BL}.

\begin{proposition} \label{prop:firstpart} For any integers $d, k \ge 2$ there  is a constructible, contractible $d$-dimensional complex 
on $k2^d + k + d$ vertices that is not ridge-chordal. 
\end{proposition}

\begin{proof}
For any $d \ge 2$, there exists a shellable contractible simplicial $d$-complex $C_d$ on $2^d+d+1$ vertices that has only one free ridge \cite{ABL}. Let $\Delta^d_{k}$ be the $d$-complex obtained by glueing together $k$ copies of the complex $C_d$ via the identification of their free ridges; the case $d=k=2$ is illustrated in Figure 1, and appeared also in Barmak's book \cite[Example 11.2.9]{BA}. 
By definition, $\Delta^d_k$ is constructible and has $k(2^d+d+1)-(k-1)d = k2^d + k + d$ vertices. By van Kampen's theorem, $\Delta^d_k$ is contractible. We claim that $\dim \Delta^d_k = \dim \operatorname{Cl}(\Delta^d_k)$. In fact, were there a face in $\operatorname{Cl}(\Delta^d_k)$ of dimension $t >d$, then $\Delta^d_k$ would contain an induced subcomplex $S$ on $t+1$ vertices with the same $d$-skeleton of the $t$-simplex. So $S$ would have nontrivial $d$-th homology, against the contractibility of $\Delta^d_k$. This proves the claim. But since $\Delta^d_k$ has no free ridge, it is neither ridge-chordal (by Lemma \ref{lem:existence1}) nor shellable (because all shellable contractible complexes are collapsible, cf.~\cite[Lemma 17]{DON}).  
\end{proof}

All minimal non-faces of $\operatorname{Cl}(\Delta^d_k)$ have dimension $d$. So the Alexander dual of $\operatorname{Cl}(\Delta^d_k)$ is pure $(k2^d+k-2)$-dimensional, with $k2^d + k + d$ vertices and $\binom{k2^{d}+d+k}{d+1}-f_d(\Delta^d_k)$ facets. 
To disprove Conjecture A, it remains to find values of $d$ and $k$ for which
the Alexander dual of $\operatorname{Cl}(\Delta^d_k)$ is shellable. Already for $d=2$ and $k=2$ this is computationally difficult, and beyond the reach of our computer. But we shall now use the theoretical trick of face-decomposability to establish shellability for $d=2$ and arbitrary $k$. 

\begin{lemma}\label{lem:VD} Let $\Delta$ be a pure simplicial complex on $[n]$. Suppose that the minimal non-faces $N_1,\ldots,N_t$ of $\Delta$ have the property that $N_j \cap N_h=\emptyset$ for every $j \neq h$. Then $\Delta$ is vertex-decomposable.
\end{lemma}

\begin{proof} 
Let $m:=\mathrm{max} \{|N_i|\}_{1 \leq i \leq t}$ and $V:=[n] \setminus \bigcup_{i=1}^t N_i$. If $m=1$, then $|N_i|=1$ for every $1 \leq i \leq t$. So

\[ \Delta= 
\begin{cases}
\{\emptyset\} & \textrm{ if } V=\emptyset\\
\textrm{a simplex} & \textrm{ if } V \ne \emptyset.
\end{cases}
\]
 
Either way, $\Delta$ is vertex-decomposable and we are done. Now suppose $m>1$ and denote by $\partial N_i$ the boundary of a simplex on the vertices of $N_i$. Then
\[ \Delta= 
\begin{cases}
\partial N_1 * \cdots * \partial N_t & \textrm{ if } V=\emptyset\\
V * \partial N_1 * \cdots * \partial N_t & \textrm{ if } V \ne \emptyset,
\end{cases}
\]

where $*$ denotes the join of simplicial complexes on disjoint sets of vertices. 
Either way, $\Delta$ is the join of vertex-decomposable complexes, hence vertex-decomposable.
\end{proof}

\begin{proof}[\bf Proof of Theorem A]
Let $k \geq 2$ and let $A_k$ be the Alexander dual of $\operatorname{Cl}(\Delta^2_k)$. Since all minimal non-faces of $\operatorname{Cl}(\Delta^2_k)$ have dimension $2$, this $A_k$ is pure $(5k-2)$-dimensional, with $n:=5k+2$ vertices and $\binom{5k+2}{3}-13k$ facets. 
Let $\gamma_j$ be the set of vertices in the $j$-th copy of $C_2$ that do not belong to the free face. Then $[n] \setminus \gamma_j$ is not in $\operatorname{Cl}(\Delta_k^2)$, because $\dim ([n] \setminus \gamma_j) =5(k-1)+1>2=\dim \operatorname{Cl}(\Delta_k^2)$. So $\gamma_j \in A_k$, for all $1 \leq j \leq k$. Define
\[ D^k_{0}:=A_k, \quad  D^k_{j}:=\operatorname{del}_{D^k_{j-1}}(\gamma_j), \quad \textrm{ and } L^k_{j}:=\operatorname{link}_{D^k_{j-1}}(\gamma_j), \quad \textrm{ for } 1 \le j \le k.\] 
If $j>1$ and $t \geq j$, we have $\gamma_t \in D_{j-1}^k$, because $\gamma_h \not \subseteq \gamma_t$, for every $h \leq j-1$. Moreover, if $k>2$, $\gamma_{j-1} \cup \gamma_j \in D_{j-2}^k$, i.e. $\gamma_{j-1} \in \operatorname{link}_{D_{j-2}^k}(\gamma_j)$, because $\dim ([n] \setminus (\gamma_{j-1} \cup \gamma_j)) =5(k-2)+1>2=\dim \operatorname{Cl}(\Delta_k^2)$.

We are going to show that $A_k$ is $4$-decomposable by induction on $k \geq 2$. Let $k=2$. We checked using  \cite{cook} that $D_1^2$ and $D_2^2$ are pure $8$-dimensional. Moreover, we checked that $L_1^2 \simeq L_2^2 \simeq A_1$, where $A_1$ is the Alexander dual of $\operatorname{Cl}(C_2)$. The reader may verify that a shelling for such $3$-complex is 
\begin{center} 
$[4,5,6,7],[3,5,6,7],[2,4,6,7],[1,4,6,7],[1,3,6,7],[1,2,6,7],[3,4,5,7],[1,3,5,7],$
$[1,2,5,7],[2,3,5,7],[2,3,4,7],[1,2,4,7],[3,4,5,6],[2,3,4,6],[2,3,5,6],[1,2,5,6],$
$[1,3,4,6],[1,2,4,6],[1,2,3,6],[1,3,4,5],[1,2,4,5],[1,2,3,4].$
\end{center}
Since $D_2^2$ is vertex-decomposable, it follows that $A_2$ is $4$-decomposable.

\noindent Now let $k>2$. Notice that $\operatorname{link}_{A_k}(\gamma_j) \simeq A_{k-1}$, for every $j$, where `$\simeq$' stands for `combinatorially equivalent'. In particular, $L_1^k \simeq A_{k-1}$. In general, we have $L^k_{j} \simeq D^{k-1}_{j-1}$. We proceed by induction on $j$. Let $j>1$. We have  $$L_j^k=\mathrm{link}_{D^k_{j-1}}(\gamma_j)=\mathrm{link}_{\mathrm{del}_{D^k_{j-2}}(\gamma_{j-1})}(\gamma_j)=\mathrm{del}_{\mathrm{link}_{D^k_{j-2}}(\gamma_{j})}(\gamma_{j-1}) \simeq \mathrm{del}_{D^{k-1}_{j-2}}(\gamma_{j-1})=D^{k-1}_{j-1},$$ where the combinatorial equivalence is ensured by $\mathrm{link}_{D^k_{j-2}}(\gamma_{j}) \simeq L_{j-1}^k \simeq D^{k-1}_{j-2}$. Moreover, the third equality holds because, for every $G \in \Delta$ and $F \in \operatorname{link}_{\Delta}(G)$, we have $\operatorname{link}_{\operatorname{del}_{\Delta}(F)}(G)=\operatorname{del}_{\operatorname{link}_{\Delta}(G)}(F)$.
We have to verify that for $j=1,2,3$, $\gamma_j$ is a shedding face of $D_{j-1}^k$. Here is a proof: 
 \begin{compactitem}
\item Let $F=[n] \setminus S$ be a facet of $A_k=D^k_0$ containing $\gamma_1$. Let $w \in \gamma_1$. We claim that there exists $s \in S$ such that $\{s,w\} \notin \Delta_k^2$. In fact, $S \cap \gamma_j \neq \emptyset$ for some $j \geq 2$, otherwise $\bigcup_{j=2}^k \gamma_j \subseteq F$. Let $r$ be the free ridge of $C_2$. Hence $S \subseteq r \cup \gamma_1$ and $S \cap \gamma_1 \neq \emptyset$, a contradiction. Let $v \in S \setminus \{s\}$ and we have $(F \setminus \{w\}) \cup \{v\} \in A_k$, because $(S \setminus \{v\}) \cup \{w\} \notin \operatorname{Cl}(\Delta_k^2)$. 
\item Let $F=[n] \setminus S$ be a facet in $D^k_1$ containing $\gamma_2$. Let $w \in \gamma_2$. Notice that $S \cap \gamma_1 \neq \emptyset$. Let $s \in S \cap \gamma_1$ and consider $v \in S \setminus \{s\}$. We have $(F \setminus \{w\}) \cup \{v\} \in D^k_1$. In fact, $(S \setminus \{v\}) \cup \{w\} \notin \operatorname{Cl}(\Delta_k^2)$, because $\{s,w\} \notin \Delta_k^2$, and $[(S \setminus \{v\}) \cup \{w\}]\cap \gamma_1 \neq \emptyset$.  
\item Let $F=[n] \setminus S$ be a facet in $D^k_2$ containing $\gamma_3$. Let $w \in \gamma_3$. Notice that $S \cap \gamma_1 \neq \emptyset$ and $S \cap \gamma_2 \neq \emptyset$. Let $s_i \in S \cap \gamma_i$, for $i=1,2$, and consider $v \in S \setminus \{s_1,s_2\}$. We have $(F \setminus \{w\}) \cup \{v\} \in D^k_2$. In fact, $(S \setminus \{v\}) \cup \{w\} \notin \operatorname{Cl}(\Delta_k^2)$, because $\{s_1,s_2\} \notin \Delta_k^2$, and $[(S \setminus \{v\}) \cup \{w\}]\cap \gamma_i \neq \emptyset$, for $i=1,2$.          
\end{compactitem}
Now we are ready to conclude. 

Since $L_j^k \simeq D^{k-1}_{j-1}$, the complexes $L_j^k$ are $4$-decomposable for $1 \leq j \leq 3$, by the inductive assumption. 
The unique minimal non-faces of $D^k_{3}$ are $\{\gamma_1,\gamma_2,\gamma_3\}$, because the set of facets of $D_3^k$ is \[\{[n] \setminus S \in A_k: |S|=3, |S \cap \gamma_j|=1, j=1,2,3\}.\] Since $\{\gamma_1,\gamma_2,\gamma_3\}$ are disjoint, then $D^k_{3}$ is vertex-decomposable by Lemma \ref{lem:VD}. Hence $A_k$ is $4$-decomposable, as desired. 
\end{proof}

\enlargethispage{3mm}

\begin{remark}\label{bigfarapproach} By the work of Bidgeli, Faridi \cite{bigfar} and Nikseresht \cite{N} there cannot be any $0$-decomposable counterexample to Conjecture \ref{conj:1}. To see this, recall that the $d$-closure of a pure $d$-dimensional simplicial complex $\Delta$ (see \cite[Definition 2.1]{bigfar}) is exactly the clique complex $\operatorname{Cl}(\Delta)$. Hence, by \cite[Proposition 2.7]{bigfar} and \cite[Theorem 3.4]{bigfar}, the following properties  are equivalent: 
\begin{compactitem}
\item $\Delta$ is ridge-chordal;
\item $\operatorname{Cl}(\Delta)$ is $d$-chordal, in the sense of Bigdeli-Faridi \cite[Definition 2.6]{bigfar};
\item $\operatorname{Cl}(\Delta)$ is $d$-collapsible, in the sense of Wegner \cite{weg}. 
\end{compactitem}
Now, let $\Delta $ be a complex such that the Alexander dual of $\operatorname{Cl}(\Delta)$ is $0$-decomposable. By \cite[Theorem 5.2]{bigfar}, the complex $\operatorname{Cl}(\Delta)$ is $d$-chordal; so by the equivalence above, $\Delta$ is ridge-chordal and Conjecture \ref{conj:1} holds.
En passant, this also explains why Conjecture \ref{conj:1} is  equivalent to \cite[Question 6.3]{bigfar}. 
Our complex $\Delta^2_2$ of Figure 1 is not ridge-chordal, so in particular $\operatorname{Cl}(\Delta^2_2)$ is not $2$-chordal. 
\end{remark}

\begin{remark}\label{cluttersapproach} In the literature, the problems we discussed are often phrased in terms of ``clutters''. Let $d \geq 1$ be an integer. A $d$-\emph{uniform clutter} $\mathcal{C}$ is the collection of the facets of a pure $(d-1)$-dimensional simplicial complex $\Gamma_{\mathcal{C}}$. Denote by $I(\mathcal{C})$ the \emph{edge ideal} of $\mathcal{C}$. Let 
$\overline{\mathcal{C}}$ be the clutter with vertices $1, \ldots, n$ whose edges are the $(d-1)$-dimensional non-faces of $\Gamma_{\mathcal{C}}$. It is easy to see  that the edge ideal of $\overline{\mathcal{C}}$ is the Stanley--Reisner ideal of $\operatorname{Cl}(\Gamma_{\mathcal{C}})$.
Moreover, the ridge-chordality of $\Gamma_{\mathcal{C}}$ is equivalent to the chordality of $\mathcal{C}$, as defined in \cite{ur}.
With this terminology, Conjecture \ref{conj:1} can be rephrased as 

\begin{center}
``\emph{For $d \geq 2$, if  $\mathcal{C}$ is a $d$-uniform clutter such that $I(\overline{\mathcal{C}})$ has linear quotients, then $\mathcal{C}$ is chordal.}''
\end{center}

Theorem \ref{main}, forgetting the constructibility and the $4$-decomposability claims, could be then stated as
\begin{center} ``\emph{Infinitely many $3$-uniform clutters $\mathcal{C}$ such that $I(\overline{\mathcal{C}})$ has linear quotients, are not chordal.}''
\end{center} 
\end{remark}

\begin{remark}
Ridge-chordality was introduced in \cite{ur} with the goal to extend  Fr\"oberg's characterization of the squarefree monomial ideals with $2$-linear resolutions \cite{froberg}. This notion was also implicit in \cite[Section 6.2]{ANS} and \cite{porto}. Several other higher-dimensional extensions of graph chordality exist in the literature: see for instance \cite{Adiprasito}, \cite{Emtander}, \cite{vt}, \cite{Wood}. A weakening of ridge-chordality is the demand that  $I(\overline{\Delta})$ have a linear resolution over any field \cite[Theorem 3.2]{ur}, where  $\overline{\Delta}$ is the complex whose facets are the $d$-dimensional non-faces of $\Delta$. As shown by \cite[Example 4.8]{bigfar} or by our complex $\Delta^2_2$ of Figure 1, some complexes satisfying this property are not ridge-chordal. 
En passant, this clarifies what is new in Proposition \ref{prop:firstpart}: 
examples of constructible and even shellable
non-ridge-chordal complexes were previously known, but they are not contractible, see for instance \cite[Exercise 7.37, pag. 277]{bj}. Examples of contractible non-ridge-chordal complexes were also known, like  \cite[Example 4.8]{bigfar}, but they are not constructible.
\end{remark}

\begin{remark}\label{nocontro}
Let $\Delta$ be a pure $d$-complex on $n+1$ vertices such that $\dim \Delta = \dim \operatorname{Cl}(\Delta)$. We claim that if the Alexander dual of $\operatorname{Cl}(\Delta)$ is shellable, then the shelling extends to  the $(n-d-1)$-skeleton of the $n$-simplex. In fact, all the minimal non-faces of $\operatorname{Cl}(\Delta)$ have cardinality $d+1$. Hence the Alexander dual $A$ of $\operatorname{Cl}(\Delta)$ has dimension $k-1$, where $k=n-d$. Moreover, the $(k-2)$-skeleton of $A$ is the $(k-2)$-skeleton of the $n$-simplex. By contradiction, let $N$ be a minimal non-face of $A$, with $|N|<k$. Then $[n+1] \setminus N$ is a facet of $\operatorname{Cl}(\Delta)$ of cardinality $|[n+1]-N|=n+1-|N|>n+1-k=d+1$ and $\dim \operatorname{Cl}(\Delta)>d$.
This implies that all the missing facets of $A$ of dimension $k-1$ can be attached along their whole boundary to extend the shelling. 
\end{remark}

\section{Open problems}
We conclude proposing   two questions:

\begin{question}
Is it true that the Alexander dual of $\operatorname{Cl}(\Delta^d_k)$ is $2^d$-decomposable?
\end{question}
 
 \begin{question}
If both $\Delta$ and the Alexander dual of $\operatorname{Cl}(\Delta)$ are shellable, is it true that $\Delta$ is ridge-chordal?
\end{question}
 
\section*{Acknowledgments} We are grateful to Anton Dochtermann and Joseph Doolittle for useful comments, and to the referees for useful corrections and references. The first author is supported by NSF Grant 1855165.

\end{document}